\documentclass[12pt, a4paper, parskip=half, abstracton]{scrartcl}

\usepackage{array}
\usepackage{marginnote}
\usepackage{xcolor}
\usepackage{amscd,amssymb,amsfonts,amsmath,latexsym,amsthm}
\usepackage{hyperref}
\usepackage[all,cmtip]{xy}
\usepackage{mathrsfs}
\usepackage{graphicx}
\usepackage{bm} 

\usepackage{etoolbox}

\def\hB{\hspace*{\fill}$\qed$}

\usepackage{chngcntr} 
\usepackage[nottoc]{tocbibind} 
\usepackage{defs_pp1}
\usepackage{slashed}
\usepackage[utf8]{inputenc}
\usepackage{microtype}
\usepackage[english]{babel}

\title{A universal coarse $K$-theory}
\author{
Ulrich Bunke\thanks{Fakult{\"a}t f{\"u}r Mathematik,
Universit{\"a}t Regensburg,
93040 Regensburg,
GERMANY\newline
ulrich.bunke@mathematik.uni-regensburg.de} 
\and
Denis-Charles Cisinski\thanks{Fakult{\"a}t f{\"u}r Mathematik,
Universit{\"a}t Regensburg,
93040 Regensburg,
GERMANY\newline
denis-charles.cisinski@mathematik.uni-regensburg.de}
}

\numberwithin{equation}{section}
\counterwithout{footnote}{section}

\newtheorem{theorem}{Theorem}[section] 
\newtheorem{prop}[theorem]{Proposition}
\newtheorem{lem}[theorem]{Lemma}
\newtheorem{ddd}[theorem]{Definition}
\newtheorem{kor}[theorem]{Corollary}

\theoremstyle{remark}
\theoremstyle{definition}

\newtheorem{ex}[theorem]{Example}
\newtheorem{rem}[theorem]{Remark}

\newcommand{\Ho}{\mathbf{Ho}}

\newcommand{\Coind}{\mathrm{Coind}}

\newcommand{\Res}{\mathrm{Res}}

\newcommand{\Orb}{\mathbf{Orb}}

\newcommand{\BC}{\mathbf{BornCoarse}}

\newcommand{\Cofib}{\mathrm{Cofib}}
\newcommand{\bB}{{\mathbf{B}}}

\newcommand{\cL}{{\mathcal{L}}}

\newcommand{\Add}{{\mathtt{Add}}}

\newcommand{\bA}{{\mathbf{A}}}

\newcommand{\cU}{{\mathcal{U}}}
\newcommand{\cY}{{\mathcal{Y}}}

 \newcommand{\Cone}{{\mathtt{Cone}}}

 \newcommand{\Cat}{{\mathbf{Cat}}}

\newcommand{\Spc}{\mathbf{Spc}}

\newcommand{\UK}{\mathrm{UK}}

\renewcommand{\Add}{\mathbf{Add}}
\newcommand{\Idem}{\mathrm{Idem}}
\newcommand{\RelCat}{\mathbf{RelCat}}

\begin{document}
\maketitle
\begin{abstract}
In this paper, we construct an equivariant coarse homology theory with values in the
category of non-commutative motives of Blumberg, Gepner and Tabuada, with coefficients in
any small additive category. Equivariant coarse $K$-theory is obtained from the latter
by passing to global sections. The present construction  extends   joint work
of the first named author
with Engel, Kasprowski and Winges 
by  promoting
  codomain of the equivariant coarse $K$-homology functor to non-commutative motives.
\end{abstract}
\tableofcontents

\section{Introduction}

In this paper we introduce a universal $K$-theory-like equivariant coarse homology theory
associated to an additive category.  

Equivariant coarse homology theories have been introduced in \cite{bekw} as the basic objects of equivariant coarse homotopy theory.   Equivariant coarse homology theories give rise to Bredon-type equivariant homology theories by a construction which we will explain in the Remark \ref{fiowfweewfewfewwfew} below. Such a presentation of equivariant homology theories  plays an important role in the proofs of isomorphism conjectures, see e.g. \cite{calped}, \cite{MR2030590},\cite{blr}\footnote{These are only some papers of a vast list  we will not try to review here.}.

As an example of an equivariant coarse homology theory, in \cite{bekw} we constructed an equivariant coarse $K$-homology functor $\mathrm{K} \cX_{\bA}^{G}$ associated to an additive category $\bA$. 
It is defined as a composition of functors \begin{equation}\label{ferflkijlerferferf}
\mathrm{K}  \cX_{\bA}^{G}:G\BC\stackrel{\bV^{G}_{\bA}}{\to} \Add\stackrel{K}{\to} \Sp\ .
\end{equation}
In this formula
 $G\BC$ denotes the category of $G$-bornological coarse spaces, and $\bV_{\bA}^{G}$ is a functor which associates to a $G$-bornological coarse space $X$ the additive category of equivariant $X$-controlled $\bA$-objects $\bV_{\bA}^{G}(X)$.
The functor $K$ is a non-connective $K$-theory functor from additive categories to spectra \cite{MR802790}.

In \cite{MR3070515} Blumberg, Gepner, and Tabuada   interpret  algebraic $K$-theory  as a localizing invariant of small stable $\infty$-categories.  They construct a universal localizing invariant \begin{equation}\label{vcoihwvoihvvwefwefwefewf}
\cU_{{loc}}:\Cat_{\infty}^{ex}\to \cM_{loc}\ ,
\end{equation}
where $\Cat^{ex}_{\infty}$ denotes the $\infty$-category of small stable $\infty$-categories.

The goal of the present paper is to construct a  factorization of  $\mathrm{K}\cX^{G}_{\bA}$ over $\cU_{{loc}}$.  To  find such a factorization is very much in spirit of  the ideas of Balmer-Tabuada \cite{Balmer:2008aa} and might be helpful in an approach to the ``Mother Isomorphism Conjecture''. It also leeds to further new examples of equivariant coarse homology theories, see Example \ref{fweoifewfewfewf}.

In Section \ref{ewfoiwefewfewfw}   we construct the  functor $\Ch^{b}(-)_{\infty}:\Add\to \Cat^{ex}_{\infty}$ which sends a small additive category $\bA$ to the small  stable $\infty$-category ${\Ch^{b}(\bA)_{\infty}}$ obtained from the bounded chain complexes over $\bA$  by inverting homotopy equivalences.
We then analyse the composition
$$\UK :  \Add\stackrel{\Ch^{b}(-)_{\infty}}{\to} \Cat^{ex}_{\infty}\stackrel{\cU_{{loc}}}{\to} \cM_{loc} \ .$$
In Section \ref{foifioufouewfwefwefw}  we first review bornological coarse spaces and define the category $\BC$. We furthermore recall the notion of an equivariant coarse homology theory,  and the definition of the categories of $X$-controlled $\bA$-objects $\bV^{G}_{\bA}(X)$. Then we continue and 
 define, in analogy to \eqref{ferflkijlerferferf}, the functor
$$\UK\cX_{\bA}^{G}: G\BC\stackrel{\bV^{G}_{\bA}}{\to} \Add \stackrel{ \UK}{\to}  \cM_{loc}\ .$$
From this functor 
 we can recover the usual coarse algebraic $K$-homology functor \eqref{ferflkijlerferferf} by  
\begin{equation}\label{fu98fu24ff23f}
\mathrm{K}\cX^{G}_{\bA}(-)\simeq \map_{\cM_{loc}}(\cU_{{loc}}(\bS^{\omega}_{\infty}),\UK\cX_{\bA}^{G}(-))\ ,
\end{equation} 
where $\bS^{\omega}_{\infty}$ denotes the small stable $\infty$-categories of compact spectra.

Our main theorem is:
\begin{theorem}[Theorem \ref{wiwofewfewfewfewfw}]\label{fjoifjoweffewfewfef} 
The functor $\UK\cX_{\bA}^{G}$ is an equivariant  coarse homology theory.
\end{theorem}

The arguments for the proof are very similar to the case of usual $K$-theory and essentially copied from \cite{bekw}.

From the universal coarse homology theory $\UK\cX_{\bA}^{G}$ associated to $\bA$ one can derive new coarse homology theories which have not been considered so far. \begin{ex} \label{fweoifewfewfewf} An example is topological Hochschild homology $\mathrm{THH}$. In \cite{MR3070515} topological Hochschild homology was characterized as a localizing invariant $$\mathrm{THH}:\Cat^{ex}_{\infty}\to \Sp\ .$$ By the universal property   of the morphism \eqref{vcoihwvoihvvwefwefwefewf} there is an essentially unique factorization 
$$\xymatrix{\Cat^{ex}_{\infty}\ar[dr]_{\cU_{{loc}}}\ar[rr]^{\mathrm{THH}}&&\Sp\\&\cM_{loc}\ar[ur]_{\mathrm{THH}_{\infty}}&}\ .$$
This allows to define an equivariant coarse homology theory
$$\mathrm{THH}\cX^{{G}}_{\bA}:G\BC\stackrel{\UK\cX_{\bA}^{G}}{\to} \cM_{loc}\stackrel{\mathrm{THH}_{\infty}}{\to}\Sp\ .$$ \hB \end{ex}

In Section \ref{fiowejofwefewfewfwf} we package the coarse homology functors
$\UK\cX_{\bA}^{H}$ for the subgroups $H$ of $G$ together into a functor
$$\underline{\UK\cX}_{\bA}:G\BC\to \Fun(\Orb(G)^{op},\cM_{loc})\ .$$

We want to stress the following point. In order to define the functor 
$\underline{\UK \cX}_{\bA}$   (Definition \ref{eiufwfwefewfewf}) it is not necessary to know how to define  equivariant coarse homology homology theories. One only  needs the  non-equivariant version  $\UK\cX_{\bA}$. Equivariance is built in using  the coinduction functor which right-Kan extends $G$-equivariant objects to functors on $\Orb(G)^{op}$. 
This idea might be helpful in other cases where the non-equivariant version of a coarse homology is already known, while the details of a construction of an equivariant versions are not yet fixed. 
The construction of the equivariant homology theories using the coinduction functor plays an important role in the proof of the descent principle discussed in  \cite{descent}.

Our main result here, besides of the construction of $ \underline{\UK\cX_{\bA}}$, is:
\begin{theorem}[Theorem \ref{wfoijwfwfewf}]
$\underline{\UK\cX_{\bA}}$ is an equivariant homology theory.
\end{theorem}

\begin{rem}The descent principle itself does not seem to apply to the universal invariants.
The main additional input of the proof of the descent principle is the fact shown by   Pedersen \cite{MR1351941} that the
connective algebraic $K$-theory functor $K_{\ge 0}:\Add\to \Sp_{\ge 0}$ preserves products.
This result extends to the non-connective $K$-theory, but probably not to the functor $\UK$.
\end{rem}

\begin{rem}\label{fiowfweewfewfewwfew}
By Elmendorff's theorem the category $\Fun(\Orb(G)^{op},\Spc)$ can be considered as the home of $G$-equivariant homotopy theory. Bredon-type $G$-equivariant $\cC$-valued homology theories   are then represented by objects in $\Fun(\Orb(G),\cC)$, see Davis-L\"uck \cite{davis_lueck}. The values of the functor $ \underline{\UK\cX_{\bA}}$ must not be confused with  such homology theories since $\underline{\UK\cX_{\bA}}$ takes values in     contravariant functors instead of covariant ones. The Davis-L\"uck type equivariant homology theories can be recovered as follows.

To every set $S$ we can associate a bornological coarse space $S_{min,max}$ given by $S$ with the minimal coarse and the maximal bornological structures. This gives a functor
$$T:\Set\to \BC\ , \quad S\mapsto S_{min,max}\ ,$$
and hence a functor
$$T^{G}:\Fun(BG,\Set)\to G\BC\ .$$
The composition
$$\UK_{\bA}^{G}:\Orb(G)\to \Fun(BG,\Set)\stackrel{T^{G}}{\to} G\BC\stackrel{\UK\cX_{\bA}^{G}}{\to}\cM_{loc}$$
is a Bredon-type $\cM_{loc}$-valued equivariant homology theory. 
The equivariant algebraic $K$ associated to an additive category as considered in
  \cite{davis_lueck}, \cite{blr}, \cite{MR2030590} can now be recovered in the style of \eqref{fu98fu24ff23f} by
  $$\mathrm{K}^{G}_{\bA}(-)\simeq \map_{\cM_{loc}}(\cU_{loc}(\bS^{\omega}_{\infty}),\UK_{\bA}^{G}(-))\ .$$ Indeed this follows from \eqref{fu98fu24ff23f} and the equivalence
  $\mathrm{K}^{G}_{\bA}(-)\simeq \mathrm{K}\cX_{\bA}^{G}\circ T^{G}$ which we check in 
  \cite{bekw}.
 \hB  \end{rem}

{\em Acknowledgement: A great part of this work is a side product of the collaboration with
A. Engel, D. Kasprovski, Ch. Winges and M. Ullmann on various projects in equivariant coarse homotopy theory.  The authors were supported by the SFB 1085 (DFG).}

\section{The universal invariant of additive categories}\label{ewfoiwefewfewfw}

\subsection{The universal localizing invariant}
Let $\Cat^{ex}_{\infty}$ denote the $\infty$-category of small stable $\infty$-categories and exact functors.
{\begin{ddd}\label{weiofuweoifwfewfwf} \mbox{}
\begin{enumerate} \item A morphism $u:\bA\to \bB$ in $\Cat^{ex}_{\infty}$ is a Morita equivalence if
{it induces an equivalence of $\Ind$-completions $\Ind(u):\Ind(\bA)\to \Ind(\bB)$.}
\item {An exact sequence in $\Cat^{ex}_{\infty}$ is a commutative square of the form
$$\xymatrix{
\bA\ar[r]^i\ar[d]&\bB\ar[d]\\
0\ar[r]&\bC
}$$}{in which the morphism $i$ is fully faithful, such that
the induced functor $\bB/\bA\to \bC$ is a Morita equivalence.}
\end{enumerate} \end{ddd}
{\begin{rem}\label{fiofwfewfewff}
For a stable $\infty$-category $\bA$, the $\Ind$-completion can be identified with the
$\infty$-category of exact functors from the opposite of $\bA$ to the stable $\infty$-category
of spectra. The idempotent completion $\Idem(\bA)$ of $\bA$
can be defined as the full subcategory of compact objects in the $\Ind$-completion of $\bA$.
Moreover, the triangulated category $\Ho(\Ind(\bA))$ is compactly generated
(more precisely, the objects of $\bA$ form a generating family of compact generators),
and the triangulated category $\Ho(\Idem(\bA))$ is canonically equivalent to the
idempotent completion (or karoubianization) of the category $\Ho(\bA)$ (in the classical sense).
Therefore, one can characterize Morita equivalences as the exact functors $\bA\to\bB$
inducing an equivalence of $\infty$-categories $\Idem(\bA)\to\Idem(\bB)$.
\end{rem}}

Following  \cite{MR3070515} a functor $\Cat^{ex}_{\infty}\to \cC$ with $\cC$ a presentable stable $\infty$-category is called a localizing invariant if it inverts Morita equivalences, sends exact sequences to fibre sequences, and preserves filtered colimits.

In \cite{MR3070515}  a universal localizing invariant
$$\cU_{{loc}}:\Cat^{ex}_{\infty}\to \cM_{loc}$$ has been constructed.
The non-connective $K$-theory of small stable $\infty$-categories is a localizing invariant.
It can be derived from the universal invariant as follows.
Let $\bS^{\omega}_{\infty}$ be the small stable $\infty$-category of compact spectra and {$\bA$} be a small stable $\infty$-category.
Then by \cite[Thm. 1.3]{MR3070515} the non-connective $K$-theory spectrum of $\bA$ is given by
{\begin{equation}\label{fewuhfuiefewfwf}
K(\bA)\simeq \map_{\cM_{loc}}(\cU_{{loc}}(\bS^{\omega}_{\infty}),\cU_{{loc}}(\bA))\ .
\end{equation}}
\begin{rem}\label{rem:generic comparison non connective K-theory}
{There is also a connective $K$-theory spectrum  $K^{Wald}(\bA)$, which is the connective cover of $K(\bA)$, and
whose value at $\pi_0$ is the usual Grothendieck group of the triangulated category $\Ho(\Idem(\bA))$.} 
\begin{ddd}\mbox{}
\begin{enumerate}
\item A stable $\infty$-category $\bA$ is flasque if there exists an exact functor $S:\bA\to \bA$
such that there exists an equivalence $\id_{\bA}\oplus S\simeq S$.
\item A delooping of a stable $\infty$-category $\bA$ is a collection of exact sequences of $\infty$-categories of the form
$$\xymatrix{
S^n(\bA)\ar[r]\ar[d]&F^n(\bA)\ar[d]\\
0\ar[r]&S^{n+1}(\bA)}$$
for {all non negative integer $n$}, such that $F^n(\bA)$ is flasque for all $n$, together with
a Morita equivalence $\bA\to S^0(\bA)$.
\end{enumerate}
\end{ddd}
 
Given a delooping of a stable $\infty$-category $\bA$ as in the definition above, we have
the following commutative squares of spectra:
\begin{equation}\label{generic comparison non connective K-theory0}
\begin{split}
\xymatrix{
K^{Wald}(S^n(\bA))\ar[r]\ar[d]&K^{Wald}(F^n(\bA))\ar[d]\\
0\ar[r]&K^{Wald}(S^{n+1}(\bA))}\ .
\end{split}
\end{equation}
Since $K^{Wald}(F^n(\bA))\simeq 0$, these squares define canonical maps
\begin{equation}\label{generic comparison non connective K-theory1}
K^{Wald}(S^n(\bA))\to\Omega(K^{Wald}(S^{n+1}(\bA)))
\end{equation}
and hence maps
\begin{equation}\label{generic comparison non connective K-theory2}
\Omega^n(K^{Wald}(S^n(\bA)))\to\Omega^{n+1}(K^{Wald}(S^{n+1}(\bA)))\ .
\end{equation}
\end{rem}
\begin{prop}\label{prop:generic comparison non connective K-theory}
Given any choice of delooping of a stable $\infty$-category $\bA$,
there is a canonical equivalence of spectra
$$\colim_{n\geq 0}\Omega^n(K^{Wald}(S^n(\bA)))\simeq K(\bA)\ .$$
Furthermore, for any non negative integer $n$, there is a canonical isomorphism
$$K_0(\Idem(S^n(\bA)))\cong \pi_{-n}{(K(\bA))}\, ,$$
where $K_0={\pi_0\circ K^{Wald}}$ denotes the Grothendieck group functor.
\end{prop}
\begin{proof}
We have cofiber sequences of the form
$$\xymatrix{
K^{}(S^n(\bA))\ar[r]\ar[d]&K^{}(F^n(\bA))\ar[d]\\
0\ar[r]&K^{}(S^{n+1}(\bA))}\ ,$$
whence equivalences
\begin{equation}\label{generic comparison non connective K-theory3}
K^{}(S^n(\bA))\to\Omega(K^{}(S^{n+1}(\bA)))
\end{equation}
from which we deduce an equivalence
\begin{equation}\label{generic comparison non connective K-theory4}
K(\bA)\xrightarrow{\simeq}\colim_{n\geq 0}\Omega^n(K^{}(S^n(\bA)))\ .
\end{equation}
By naturality of the map $K^{Wald}\to K$, we get the commutative square
$$\xymatrix{
K^{Wald}(\bA)\ar[r]\ar[d]&\colim_{n\geq 0}\Omega^n(K^{Wald}(S^n(\bA)))\ar[d]\\
K(\bA)\ar[r]^(.3){\simeq}&\colim_{n\geq 0}\Omega^n(K^{}(S^n(\bA)))
}$$
of which any inverse of the bottom horizontal map gives a map
\begin{equation}\label{generic comparison non connective K-theory5}
\colim_{n\geq 0}\Omega^n(K^{Wald}(S^n(\bA)))\to K(\bA)
\end{equation}
Since $K^{Wald}(S^n(\bA))$ is the connective cover of $K(S^n(\bA))$,
the canonical map from $K^{Wald}_j(S^n(\bA))$ to $K_j(S^n(\bA)) \cong  K_{j-n}(\bA)$
is an isomorphism for all non negative integers $j$.
Since the formation of stable homotopy groups commutes with small
filtered colimits, we have canonical isomorphisms of groups for any integer $i$:
{\begin{align*}
\pi_i\big(\colim_{n\geq 0}\Omega^n(K^{Wald}(S^n(\bA)))\big)
&\cong\colim_{n\geq 0}\pi_{i+n}(K^{Wald}(S^{n}(\bA)))\\
&\cong\colim_{n\geq -i}\pi_{i+n}(K^{Wald}(S^{n}(\bA)))\\
&\cong\colim_{n\geq -i}\pi_{i+n}(K(S^{n}(\bA)))\\
&\cong\colim_{n\geq 0}\pi_{i+n}(K(S^{n}(\bA)))\\
&\cong\pi_i\big(\colim_{n\geq 0}\Omega^{n}(K(S^{n}(\bA)))\big)\\
&\cong K_i(\bA)\ .
\end{align*}}Therefore, map \eqref{generic comparison non connective K-theory5}
is a stable weak homotopy equivalence.
\end{proof}}

\subsection{A universal $K$-theory of additive categories}

An additive category $\bA$ can be considered as a Waldhausen category whose weak equivalences are the isomorphisms and whose cofibrations are the inclusions of direct summands. The connective $K$-theory {$K^{Wald}(\bA)$} is defined as the Waldhausen $K$-theory of this Waldhausen category
{(this agrees with Quillen's definition, up to a functorial equivalence)}.
{This} construction can be refined in order to produce a non-connective $K$-theory spectrum $K(\bA)$ such that
$\pi_{i}(K^{}(\bA))\cong \pi_{i}(K^{Wald}_{}(\bA))$ for all integers $i$, with $i\ge 1$ (or $i\ge 0$ if $\bA$ is idempotent complete); {see \cite{MR802790,MR2206639}.}
{In Proposition \ref{prop:agreement} below we provide an alternative description as a specialization of a universal $K$-theory for additive categories.}

We consider the categories $\Add$ and $\RelCat$ of small additive categories and relative categories. We have a functor  \begin{equation}\label{ecfiuwhiuwefwfwe}(\Ch^{b}(-),W_{h}):\Add\to \RelCat \end{equation}which sends an additive category $\bA$  to the relative category $(\Ch^{b}(\bA),W_{h})$ of bounded chain complexes and homotopy equivalences. We have a localization functor \begin{equation}\label{fiuzhriuferferferf}
 \cL:\RelCat\to \Cat_{\infty}  \end{equation} which sends a relative category $(\bC,W)$ to the localization $\bC[W^{-1}]$. For an additive category $\bA$ we will use the notation 
$$\Ch^{b}(\bA)_{\infty}:=\cL(\Ch^{b}(\bA),W_{h})\ .$$

\begin{rem}\label{oijoifwefewfwefwf} Assume that we model $\infty$-categories by quasi-categories. Then a model for the localization functor is given by
$$(\bC,W)\mapsto \Nerve((L^{H}_{W}\bC)^{fib})\ ,$$
where $L^{H}_{W}\bC$ is the hammock localization producing a simplicial category from a relative category, $(-)^{(fib)}$ is the fibrant replacement in simplicial categories, and $\Nerve$ is the coherent nerve functor which sends fibrant simplicial categories to quasi categories. 
\end{rem}

As explained for instance in \cite[Section 4]{MR1469141},
the category $\Ch^{b}(\bA)$ has the structure of a cofibration category where
\begin{enumerate}
\item weak equivalences are homotopy equivalences,
\item cofibrations are morphisms of chain complexes which are degree-wise split injections.
\end{enumerate}
Note that if $A\to B$ is a cofibration in $\Ch^{b}(\bA)$, then  the quotient $A/B$ exists.
In this case we say that $A\to B\to A/B$ is a short exact sequence.

Let \begin{equation}\label{d23uifzi2uf2f3f}
\ell:\Ch^{b}(\bA)\to \Ch^{b}(\bA)_{\infty}
\end{equation} denote the canonical morphism.

\begin{prop}\label{fweopfwefewfwef}
If $\bA$ is a small  additive category, then $\Ch^{b}(\bA)_{\infty}$
is a small stable $\infty$-category.
{Moreover, 
\begin{enumerate} \item  \eqref{d23uifzi2uf2f3f} 
 preserves the null object. \item  \eqref{d23uifzi2uf2f3f} sends 
 short exact sequences  to cofiber sequences. \end{enumerate}}
\end{prop}
\begin{proof}
Let $W_{ch}$ denote the subset of $W_{h}$ of  homotopy equivalences which are in addition cofibrations. 
\begin{lem}\label{fweiowefwefwef}
The morphism of relative categories $(\Ch^{b}(\bA),W_{ch})\to (\Ch^{b}(\bA),W_{h})$ induces an equivalence  
 $\cL(\Ch^{b}(\bA),W_{ch})\to \Ch^{b}(\bA)_{\infty}$
 of $\infty$-categories.
\end{lem}\begin{proof} Since $\Ch^{b}(\bA)$ is a cofibration category, by Ken Brown's Lemma every morphism $f:A\to B$ in $\Ch^{b}(\bA)$ has a factorization 
$$\xymatrix{A\ar[rr]^{f}\ar[dr]^{i}&&B\ar@/^0.5cm/@{-->}[dl]^{r}\\&C\ar[ur]^{p,\simeq}}\ ,$$
where $i$ is a cofibration, $p$ is a weak equivalence, and furthermore $r$ is a cofibration such that $p\circ r=\id_{B}$. If $f$ is in addition a weak equivalence, then so are $i$ and $r$.  So every functor from $\Ch^{b}(\bA)$ to an infinity category which sends the elements of $W_{ch}$ to equivalences also sends the elements of $W_{h}$ to equivalences. This implies the assertion in view of {the universal property of the functor
\eqref{d23uifzi2uf2f3f}}.\end{proof}

By \cite[Thm. 2.2.0.1 and Prop. 2.2.4.1]{htt} and Lemma \ref{fweiowefwefwef}
the mapping spaces of the  $\infty$-category $\Ch^{b}(\bA)_{\infty}$
are represented by the simplicial mapping sets of the 
 hammock localization
$(L^{H}_{W_{ch}}\Ch^{b}(\bA))^{fib}$, see Remark \ref{oijoifwefewfwefwf}. More precisely,   for objects $A$ and $B$ of $\Ch^{b}(\bA)$ we have a canonical homotopy equivalence {\begin{equation}\label{frefojfoiff2f}
\Map_{\Ch^{b}(\bA)_{\infty}}(\ell(A),\ell(B))\simeq \Map_{(L^{H}_{W_{ch}}\Ch^{b}(\bA))^{fib}}(A,B) \ .
\end{equation}}The homotopy type of the simplicial mapping sets in the hammock localization   can be computed 
using a method first introduced by Dwyer and Kan, and further studied by Weiss in \cite[Rem. 1.2]{weiss}.
For two objects $A$ and $B$ of $\Ch^{b}(\bA)$ we consider the following category
$\cM(A,B)$:\begin{enumerate}
\item The objects are pairs $(f,s)$, where
$f:A\to C$  is just a morphism and $s:B\to C$ a trivial cofibration.
\item A morphism  $u:(f,s)\to (f',s')$ is a commutative diagram
of the form
$$\xymatrix@R=8pt@C=25pt{
&C\ar[dd]^u&\\
A\ar[ur]^f\ar[dr]_{f'}&&B\ar[ul]_s\ar[dl]^{s'}\\
&C'&
}\ .$$ \item The composition is defined in the obvious way.
\end{enumerate}
Then by  \cite[Rem. 1.2]{weiss} we have a canonical weak equivalence of simplicial sets
\begin{equation}\label{lfjoijfoif24f2f2}
\Nerve(\cM(A,B))\simeq \Map_{(L^{H}_{W_{ch}}\Ch^{b}(\bA))^{fib}}(A,B)\ .
\end{equation}

We now consider a pushout square in $\Ch^{b}(\bA)$ of the form
\begin{equation}\label{eq:fweopfwefewfwef1}
\begin{split}
\xymatrix{
A\ar[d]_i\ar[r]^a&A'\ar[d]^{i'}\\
B\ar[r]^b&B'
}\end{split}
\end{equation}
in which the map $i$ is a cofibration. By virtue of \cite[Thm. 2.1]{weiss},
for any object $E$ in $\Ch^{b}(\bA)$, the commutative square of simplicial sets
\begin{equation}\label{eq:fweopfwefewfwef2}
\begin{split}
\xymatrix{
\Nerve(\cM(A,E))&\Nerve(\cM(A',E))\ar[l]\\
\Nerve(\cM(B,E))\ar[u]&\Nerve(\cM(B^{\prime},E))\ar[u]\ar[l]
}\end{split}
\end{equation}
is homotopy Cartesian. 
This implies, by \eqref{lfjoijfoif24f2f2}, that{\begin{equation}\label{eq:fweopfwefewfweef2}
\begin{split}
\xymatrix{
\Map_{(L^{H}_{W_{ch}}\Ch^{b}(\bA))^{fib}}(A,B)(A,E)&\Map_{(L^{H}_{W_{ch}}\Ch^{b}(\bA))^{fib}}(A,B)(A',E)\ar[l]\\
\Map_{(L^{H}_{W_{ch}}\Ch^{b}(\bA))^{fib}}(A,B)(B,E)\ar[u]&\Map_{(L^{H}_{W_{ch}}\Ch^{b}(\bA))^{fib}}(A,B)(B^{\prime},E)\ar[u]\ar[l]
}\end{split}
\end{equation}}is {a} homotopy Cartesian square of Kan complexes.
Consequently, by virtue of \cite[Thm. 4.2.4.1]{htt} (and possibly \cite[Remark A.3.3.13]{htt}),
the square \eqref{eq:fweopfwefewfwef1}
is sent by $\ell$ to a pushout square in $\Ch^{b}(\bA)_{\infty}$.

Similarly, for every object $A$ of $\Ch^{b}(\bA)$ the simplicial set
$\Nerve(\cM(A,0))$  is contractible since it is 
the nerve of a category with an initial object. Hence $\ell$ sends
the null complex to a terminal object in $\Ch^{b}(\bA)_{\infty}$.

Note that the opposite of an additive category is again an additive category.
Appying what precedes to the opposite category of $\bA$,
we see that $\ell$ sends $0$ also to  an initial object in $\Ch^{b}(\bA)_{\infty}$,
and that any pullback square of the form 
\eqref{eq:fweopfwefewfwef1} in which the map $b$
is a degree-wise split surjection is sent   to a  pull-back
square in $\Ch^{b}(\bA)_{\infty}$.

The explicit description   of the mapping spaces in $\Ch^{b}(\bA)_{\infty}$ given by \eqref{frefojfoiff2f} and \eqref{lfjoijfoif24f2f2}
implies that, up to equivalence, all maps
of $\Ch^{b}(\bA)_{\infty}$ come from maps of $\Ch^{b}(\bA)$.
In particular, by virtue of \cite[Cor. 4.4.2.4]{htt} and its dual
version, we have proved that the $\infty$-category
$\Ch^{b}(\bA)_{\infty}$ has finite limits and finite colimits, and that
$0=\ell(0)$ is a null object. 

Since, up to equivalence, any
morphism in  $\Ch^{b}(\bA)_{\infty}$ is the image under $\ell$ of a cofibration,   any cofiber sequence of $\Ch^{b}(\bA)_{\infty}$ is the image
under  $\ell$ of a pushout square of the form
$$\xymatrix{
A\ar[r]^i\ar[d]&B\ar[d]^p\\
0\ar[r]&C
}$$
in which $i$ is a  cofibration.
But such a square also is a pullback square, with $p$ a degree-wise
split surjection, and therefore, the functor $\ell$ sends such a square to
a fiber sequence. In other words, any cofiber sequence in $\Ch^{b}(\bA)_{\infty}$
is a fiber sequence. Replacing $\bA$ by its opposite category,
we see that we also proved the converse: any fiber sequence in $\Ch^{b}(\bA)_{\infty}$
is a cofiber sequence. In other words, the $\infty$-category $\Ch^{b}(\bA)_{\infty}$
is stable in the sense of \cite[Def. 1.1.1.9]{halg}, and we have proved the proposition.
\end{proof}

\begin{rem}\label{foiuwoiefwefwefewfew}
For an additive category $\bA$ the category $\Ch^{b}(\bA)$ has a natural $dg$-enriched refinement $\Ch^{b}(\bA)_{dg}$.  By  \cite[Thm. 1.3.1.10]{halg} we can thus form a small $\infty$-category $\Nerve_{dg}(\Ch^{b}(\bA)_{dg})$.  It is easy to check that
this   category is pointed by the zero complex, has finite colimits (sums and push-outs exist), and that the suspension is represented by the shift. Consequently, $\Nerve_{dg}(\Ch^{b}(\bA)_{dg})$ is a stable $\infty$-category. In fact, \cite[1.3.2.10]{halg} asserts this for the  $dg$-nerve $\Nerve_{dg}(\Ch(\bA)_{dg})$ of the $dg$-category of not necessarily bounded chain complexes. We can consider $\Nerve_{dg}(\Ch^{b}(\bA)_{dg})$ as a full subcategory of $\Nerve_{dg}(\Ch(\bA)_{dg}) $  which is stable under finite colimits.

By the universal property of the localization functor \eqref{fiuzhriuferferferf} the natural morphism
$\Ch^{b}(\bA)\to \Nerve_{dg}(\Ch^{b}(\bA)_{dg})$ factorizes through a morphism
\begin{equation}\label{fweoifwoie2455t3545}
\Ch^{b}(\bA)_{\infty}\to \Nerve_{dg}(\Ch^{b}(\bA)_{dg})\ .
\end{equation}
This morphism induces an equivalence of homotopy categories
$$\Ho(\Ch^{b}(\bA)_{\infty})\stackrel{\simeq}{\to} \Ho(\Nerve_{dg}(\Ch^{b}(\bA)_{dg}))\ .$$
Proposition \ref{fweopfwefewfwef} is now equivalent to the fact that \eqref{fweoifwoie2455t3545} is an equivalence of $\infty$-categories.
\hB
\end{rem}

Let $A$ be an object of $\Ch^{b}(\bA)$.
\begin{kor}\label{fewiowefwefewfwf}
{We have the relation $\Sigma(\ell(A))\simeq \ell(A[1])$.}
\end{kor}

\begin{proof}{ Let $A$ be an object of $\Ch^{b}(\bA)$. By Proposition \ref{fweopfwefewfwef},
the functor $\ell$ sends
  the short exact sequence
 $$A\to \Cone(\id_{A})\to A[1]$$
  to  the cofiber sequence 
{$$\xymatrix{
\ell(A)\ar[r]\ar[d]&\ell({\Cone(\id_{A}))}\ar[d]\\
0\ar[r]&\ell({A[1]})
}$$}Since $0\to \Cone(\id_{A})$ is a chain homotopy equivalence, this 
cofiber sequence exhibits $\ell(A[1])$ as the suspension of $\ell(A)$} in $\Ch^{b}(\bA)_{\infty}$.\end{proof}
\begin{prop}
The functor
$\Ch^{b}(-)_{\infty} $ has a natural factorization
$$\xymatrix{&&\Cat^{ex}_{\infty}\ar[d]\\\Add\ar[rr] \ar@{-->}[urr]^{\Ch^{b}(-)_{\infty}} && \Cat_{\infty}}\ .$$
\end{prop}
\begin{proof}
Exactness of functors and stability of $\infty$-categories are detected in $\Ho(\Cat_{\infty})$.
We let $\bE$ be  be the non-full subcategory of  $\Ho(\Cat_{\infty})$
consisting of stable $\infty$-categories and exact functors. Then we have a pull-back
$$\xymatrix{\Cat^{ex}_{\infty}\ar[r]\ar[d]&\Cat_{\infty}\ar[d]\\\bE\ar[r]&\Ho(\Cat_{\infty})}\ .$$
By Proposition \ref{fweopfwefewfwef}, the functor $\Ch^{b}(-)_{\infty}$ takes values in stable $\infty$-categories.
We must show that it send a morphism  $f:\bA\to \bA^{\prime}$ between additive categories to an exact functor. 
By {\cite[Cor. 1.4.2.14]{halg}} we must show that $\Ch^{b}(f)_{\infty}$ preserves the zero object and suspension.
It clearly preserves the zero object (represented by the zero complex). By Corollary \ref{fewiowefwefewfwf} in the domain and target of $\Ch^{b}(f)_{\infty}$
the suspension is represented by the shift and $\Ch^{b}(f)$ clearly preserves the shift,   $\Ch^{b}(f)_{\infty}$ also preserves suspension. 
\end{proof}

%
\begin{ddd}\label{coiwehcoiwcwecw}
We define the  functor $\UK:\Add\to \cM_{loc}$ as the composition
{$$\UK:\Add\stackrel{\Ch^{b}(-)_{\infty}}{\to}  \Cat^{ex}_{\infty}\stackrel{\cU_{loc}}{\to} \cM_{loc}\ .$$}
\end{ddd}
 
%

\subsection{Properties of the universal $K$-theory for additive categories}

Recall that a morphism between additive categories is a functor between the underlying
categories which preserves finite coproducts. It is an equivalence if the underlying morphism of categories is an equivalence. 
 More generally,  it is a Morita equivalence if it induces an equivalence of idempotent completions.
Finally, recall  the similar Definition \ref{weiofuweoifwfewfwf} (in conjuction with Remark \ref{fiofwfewfewff}) for  morphisms between stable $\infty$-categories.

Let $\bA$ be an additive category.
\begin{lem}\label{fowipfwfewfewf}
The morphism  $\bA\to \Idem(\bA)$ induces a Morita equivalence
$$\Ch^{b}(\bA)_{\infty}\to \Ch^{b}(\Idem(\bA))_{\infty}\ .$$
\end{lem}
\begin{proof}
In view of Definition \ref{weiofuweoifwfewfwf} and Remark \ref{fiofwfewfewff}  we must show that the induced morphism
$$c:\Idem(\Ch^{b}(\bA)_{\infty})\to \Idem( \Ch^{b}(\Idem(\bA))_{\infty})$$
is an equivalence. We first observe that for any additive category $\bB$ the natural morphism $\bB\to \Idem(\bB)$ induces a fully faithful embedding
$\Ch^{b}(\bB)_{\infty}\to \Ch^{b}(\Idem(\bB))_{\infty}$. This is obvious if we use the description of the $\infty$-categories of chain complexes via $dg$-nerves as discussed in Remark    \ref{foiuwoiefwefwefewfew}.
We now have a commuting diagram 
$$\xymatrix{\Ch^{b}(\bA)_{\infty}\ar[r]\ar[d]&\Ch^{b}(\Idem(\bA))_{\infty}\ar[d]^{i}\\\Idem(\Ch^{b}(\bA)_{\infty})\ar[r]^(0.4){c}
&\Idem(\Ch^{b}(\Idem(\bA))_{\infty})}\ .$$
Since the vertical morphisms and the upper horizontal morphism are fully-faithful, so is $c$.
It remains to show that $c$ is essentially surjective. 
We note that the morphism $i$ identifies  $\Ch^{b}(\Idem(\bA))_{\infty}$ with the smallest full stable  subcategory of $ \Idem( \Ch^{b}(\Idem(\bA))_{\infty})$  containing the   objects in the image of the composition \begin{equation}\label{f3498u9348f34f}
\Idem(\bA)\stackrel{[0]}{\to} \Ch^{b}(\Idem(\bA))\stackrel{\ell}{\to} \Ch^{b}(\Idem(\bA))_{\infty}\ .
\end{equation} 
In view of the commuting diagram $$\xymatrix{&\Ch^{b}(\Idem(\bA))_{\infty}\ar[dr]^{i}&\\\Idem(\bA)\ar[ur]^{\eqref{f3498u9348f34f}}\ar[dr]^{\Idem(\ell\circ [0])}&& \Idem( \Ch^{b}(\Idem(\bA))_{\infty})\\&\Idem(\Ch^{b}(\bA)_{\infty})\ar[ur]^{c}&
 }$$ every object in  the image of the composition  $i\circ \eqref{f3498u9348f34f}$ also belongs to the essential image of $c$.  Hence $c$ is essentially surjective and we have shown the lemma \end{proof}

{\begin{rem}
{A particular case of a result of
Balmer and Schlichting \cite[Thm. 2.8]{MR1813503}  states} that,
for any small idempotent complete additive category $\bA$, the triangulated
category $\Ho(\Ch^b(\bA)_\infty)$ is idempotent complete.
Applying this result to $i$ and  the proof of Lemma \ref{fowipfwfewfewf} to $c$ we conclude that
 for any small additive category $\bA$, the natural functors
$$\Idem(\Ch^b(\bA)_\infty)\stackrel{c}{\to}\Idem(\Ch^b(\Idem(\bA))_\infty)\stackrel{i}\leftarrow\Ch^b(\Idem(\bA))_\infty$$
are equivalences of $\infty$-categories.
\end{rem}}

 
\begin{lem}\label{roirujoigegergergg}
{The functor $\UK:\Add\to \cM_{{loc}}$ sends Morita
equivalences of additive categories to equivalences in $ \cM_{loc}$.}
\end{lem}
\begin{proof}
{Since the functor $\cU_{loc}$ sends Morita equivalences of stable $\infty$-categories
to equivalences of spectra, it is sufficient to prove that
the functor $\bA\mapsto\Ch^b(\bA)_\infty$ sends Morita equivalences of
additive categories to Morita equivalences of $\infty$-categories.}

First of all, this functor sends equivalences of additive categories to equivalences
{of $\infty$-categories} because the functor {$\bA\mapsto\Ch^b(\bA)$} preserves
equivalences of categories, and because the localization
functor \eqref{fiuzhriuferferferf} sends equivalences
of {(relative) categories} to equivalences of $\infty$-categories.

For an additive category $\bA$ the morphism $\bA\to \Idem(\bA)$ is a Morita equivalence of additive categories. 
By Lemma \ref{fowipfwfewfewf}  we get a Morita equivalence
$\Ch^{b}(\bA)_{\infty}\to \Ch^{b}(\Idem(\bA))_{\infty}$ of stable $\infty$-categories.

Any Morita equivalence $f:\bA\to\bA'$ between additive categories
fits into a commutative square of the form
$$\xymatrix{
\bA\ar[r]^{f}\ar[d]&\bA'\ar[d]\\
\Idem(\bA)\ar[r]^{\Idem(f)}&\Idem(\bA')
}$$
in which the functor $\Idem(f)$ is an equivalence of categories.
By the observations made above the functor {$\Ch^b(-)_\infty$} sends the two vertical and the lower horizontal morphism to equivalences.
This readily implies that the map $\UK(f)$ is an equivalence {from $\Ch^b(\bA)_\infty$
to $\Ch^b(\bA')_\infty$.}
 \end{proof}

Let $\bA$ be an additive category. 
\begin{ddd}\label{wefifjweifoefwefwff}
{A full additive subcategory $\bC$ of $\bA$ is a Karoubi-filtration if every diagram
$$X\to Y\to Z$$ in $\bA$ with $X,Z\in \bC$} admits an extension
 $$\xymatrix{X\ar[d]\ar[r]&Y\ar[r]\ar[d]^{\cong}&Z \\U\ar[r]^{incl}&U\oplus U^{\perp}\ar[r]^{\pr}&U\ar[u]}$$
with $U\in \bC$.
\end{ddd}{In \cite[Lemma 5.6]{KasFDC} it is shown that  Definition \ref{wefifjweifoefwefwff} is equivalent to the standard definition of a Karoubi filtration as considered in  \cite{MR1469141}.}

A morphism in $\bA$ is called completely continuous (c.c.) if it factorizes over a morphism  in $\bC$. 
We can form a new additive category $\bA/\bC$ with the same objects as $\bA$ and the morphisms
$$\Hom_{\bA/\bC}(A,A^{\prime}):=\Hom_{\bA}(A,A^{\prime})/c.c.\ .$$

 Assume that $\bC$ is a Karoubi filtration in an additive category $\bA$.
     \begin{lem}\label{riowgrgergggergerg} 
The sequence of stable $\infty$-categories
$$\Ch^{b}( \bC)_{\infty}\to \Ch^{b}( \bA)_{\infty}\to \Ch^{b}( \bA/ \bC)_{\infty}$$
is exact.
\end{lem}
\begin{proof}
{As observed in Lemma \ref{fowipfwfewfewf},
for any additive category $\bB$, the natural {functor}
$\Ch^{b}(\bB)_{\infty}\to\Ch^{b}(\Idem(\bB))_{\infty}$
is a Morita equivalence of stable $\infty$-categories.}
{Furthermore the canonical functor
$\Idem(\bA/\bC)\to\Idem(\Idem(\bA)/\Idem( \bC))$
is an equivalence of categories,
because the functor $\Idem$ (being a left adjoint)
sends pushouts of additive categories to pushouts in the
category of idempotent complete additive categories.
Therefore, the canonical functor
$$\Ch^b(\bA/\bC)_{{\infty}}\to\Ch^b(\Idem(\bA)/\Idem(\bC))_{{\infty}}$$
is a Morita equivalence of stable $\infty$-categories.}

{All this means that, without loss of generality, it is
sufficient to prove the case where $\bC$
is idempotent complete (since we may replace $\bA$ and $\bC$
by their idempotent completions at will).}
By \cite[Prop. 5.15]{MR3070515} it suffices to show that,
{under this extra assumption on $\bC$},
$$\Ho(\Ch^{b}(\bC)_{\infty})\to \Ho( \Ch^{b}(\bA)_{\infty})\to \Ho(\Ch^{b}(\bA/\bC)_{\infty})$$
is exact. {In view of Example \cite[Ex. 1.8]{MR2079996}
this is shown in \cite[Prop. 2.6]{MR2079996}.} \end{proof}

Assume that $\bC$ is a Karoubi filtration in an additive category $\bA$.
\begin{prop}\label{roijoregergegr}
We have the following fibre sequence in $\cM_{loc}$.
$$\UK(\bC)\to \UK(\bA)\to  \UK(\bA/\bC)$$
\end{prop}
\begin{proof}
%
{This immediately follows from {Lemma \ref{riowgrgergggergerg}, Definition \ref{coiwehcoiwcwecw},
and the fact that $\cU_{loc}$ sends exact sequences to fibre sequences.}} \end{proof}

{\begin{ddd}
An additive category $\bA$ is called flasque if there exists an  endo-functor $S:\bA\to \bA$   such that {$\id_{\bA}\oplus S\simeq S$.}
\end{ddd}}

 Let $\bA$ be a small additive category. By $K(\bA)$, we denote  Schlichting's non-connective $K$-theory spectrum. Its  stable homotopy groups in non-negative degrees $i$ agree with Quillen's $K$-theory groups of $\Idem(\bA)$  denoted above by  $K^{Wald}_i(\Idem(\bA))$, and with the  $K$-theory groups defined by Pedersen and Weibel in negative degrees.
 
 Let $\bA$ be a small additive category. 
 {\begin{prop}\label{prop:agreement}
There is a canonical equivalence of spectra
$$K(\bA)\simeq\map_{\cM_{loc}}(\cU_{{loc}}(\bS^{\omega}_{\infty}),\UK(\bA))\ .$$
\end{prop}
\begin{proof}
{Using Lemma \ref{fowipfwfewfewf} and the fact that $\cU_{loc}$ sends Morita equivalences to equivalences,
Formula \eqref{fewuhfuiefewfwf} expresses $\map_{\cM_{loc}}(\cU_{{loc}}(\bS^{\omega}_{\infty}),\UK(\bA))$
as the non-connective $K$-theory spectrum $K(\Ch^b(\Idem(\bA))_\infty)$.
By virtue of \cite[Cor. 7.12]{MR3070515}
the connective cover of $K(\Ch^b(\Idem(\bA))_\infty)$ agrees with Waldhausen's K-theory of $\Ch^b(\Idem(\bA))$,
which in turns agrees with Quillen's $K$-theory of $\Idem(\bA)$ by
the Gillet-Waldhausen theorem (as shown in \cite{TT}).
There is a construction, due to Karoubi, and recalled in \cite[Example 5.1]{MR802790},
which associates functorially, to any small additive category $\bA$,
a flasque additive category $C(\bA)$, together with an additive full embedding
$\bA\to C(\bA)$ which turns $\bA$ into a Karoubi filtration of $C(\bA)$.
The suspension $S(\bA)$ of $\bA$ is then defined as the quotient $C(\bA)/\bA$.
Iterating this construction, we see with Lemma \ref{riowgrgergggergerg}
that we have a delooping of the stable $\infty$-category $\Ch^b(\bA)_\infty$,
of the form
$$\xymatrix{
\Ch^b(S^n(\bA))_\infty\ar[r]\ar[d]&\Ch^b(C(S^n(\bA)))_\infty\ar[d]\\
0\ar[r]&\Ch^b(S^{n+1}(\bA)_\infty{)}
}$$
This delooping exists at the level of Waldhausen categories
$$\xymatrix{
\Ch^b(S^n(\bA))\ar[r]\ar[d]&\Ch^b(C(S^n(\bA)))\ar[d]\\
0\ar[r]&\Ch^b(S^{n+1}(\bA){)}
}$$
By naturality of the comparison maps $K^{Wald}(\Ch^b(\bA))\to K^{Wald}(\Ch^b(\bA))$
(where the left hand denotes the usual Waldhausen construction associated
with the cofibration category $\Ch^b(\bA)$),
we obtain a natural morphism of spectra:
$$\colim_{n\geq 0}\Omega^n(K^{Wald}(\Ch^b(S^n(\bA)))\to
\colim_{n\geq 0}\Omega^n(K^{Wald}(\Ch^b(S^n(\bA)))_\infty)\, .$$
By virtue of Proposition \ref{prop:generic comparison non connective K-theory},
the codomain of this map is the non-connective spectrum $K(\Ch^b(\bA)_\infty)$.
The domain of this map is precisely Schlichting's definition \cite[Def. 8]{MR2206639}
of the non-connective
$K$-theory of $\bA$. Since both side agree canonically on connective covers,
to prove that this comparison map is a stable weak homotopy equivalence,
it is sufficient to prove that it induces an isomorphism on stable homotopy
groups in non positive degree. But this map induces an isomorphism
on $\pi_0$ for any $\bA$. Therefore, replacing $\bA$ by $S^n(\bA)$, for positive integers $n$,
we see that it induces an isomorphism in degree $-n$, because both sides
are canonically isomorphic to $$K_{-n}(\bA)=K_0(\Idem(S^n(\bA)))\cong K_0(\Idem(\Ch^b(S^n(\bA))))$$
(this is  proved in Theorem \cite[Theorem 8, p. 124]{MR2206639}
for the left-hand side, and follows from the last assertion of
Proposition \ref{prop:generic comparison non connective K-theory}
for the right-hand side).}
\end{proof}

Let $\bA$ be an additive category.

    \begin{prop}\label{eroigerggerger}
If $\bA$ is flasque, then $\UK(\bA)\simeq 0$.
\end{prop}
\begin{proof}
 
 Let $\bB$ be an additive category.  We  define a new additive category
$\bD(\bB)$ by:
\begin{enumerate}
\item An object of $\bB$ is a triple $(B,B_{0},B_{1})$ of an object $B$ of $\bB$ with two subobjects such that the induced morphism $B_{0}\oplus B_{1}\to B$ is an isomorphism.
\item A morphism $(B,B_{0},B_{1}) \to (B^{\prime},B^{\prime}_{0},B^{\prime}_{1})$ is  a morphism  $B\to B^{\prime}$ preserving the subobjects.
\end{enumerate} 

We have a sequence
$$\bB\stackrel{i}{\to} \bD(\bB)\stackrel{p}{\to} \bB$$
where the first morphism $i$ sends the object $B$ in $\bB$ to the  object $(B,B,0)$ of $\bD(\bB)$, and the second morphism $p$ sends the object $(B,B_{0},B_{1})$ of $\bD(\bB)$ to  the object $B_{1}$ of $\bB$. 
The   morphism $i$ is fully faithful and determines a 
Karoubi filtration on $\bD(\bB)$, and the morphism $p$ is the projection onto the quotient.  
It has a left-inverse $j:\bB\to \bD(\bB)$   given by $B\mapsto (B,0,B)$.

We finally  have a functor $s:\bD(\bB)\to \bB$ which sends $(B,B_{0},B_{1})$ to $B$.

Let $f,g:\bA\to \bB$ be two  morphisms between additive categories.
Then  we can define a morphism $(f,g):\bA\to \bD(\bB)$ by
$A\mapsto (f(A)\oplus g(A),f(A),g(A))$ which is uniquely determined by the choice of representatives of the sums. We now consider the composition
$$h:\bA\stackrel{(f,g)}{\to} \bD(\bB)\stackrel{s}{\to} \bB\ .$$  
 Note that $\cM_{loc}$ is stable so that it makes sense to add morphisms.
We have $$\UK(f),\UK(g),\UK(h)\in \pi_{0}(\Map_{\cM_{loc}}(\UK(\bA),\UK(\bB)))\ .$$ \begin{lem}
We have $\UK(h)\simeq \UK(f)+\UK(g)$.
\end{lem}
\begin{proof}
We have a fibre sequence
$\UK(\bB)\stackrel{\UK(i)}{\to} \UK(\bD(\bB))\stackrel{\UK(p)}{\to} \UK(\bB)$ which is split by
$\UK(j)$.
We have a {commutative} diagram
{$$\xymatrix{\UK(\bA)\ar@/^1.7cm/[rrrrrr]^{\UK(f)+\UK(g)}\ar[rr]^(.42){\UK(f)\oplus \UK(g)}\ar@{=}[d]&&\UK(\bB)\oplus \UK(\bB)\ar[rrr]^(.53){\UK(i)\circ \pr_{0}+ \UK(j)\circ \pr_{1}}_{\simeq}\ar@{-->}[d]&&&\UK(\bD(\bB))\ar@{=}[d]\ar[r]^{\UK(s)}&\UK(\bB)\ar@{=}[d]\\\UK(\bA)\ar@/_{1.7cm}/[rrrrrr]^{\UK(h)}\ar[rr]^{\UK(f,g)}&&\UK(\bB\oplus \bB)\ar[rrr]^{\UK(i,j)}&&&\UK(\bD(\bB))\ar[r]^{\UK(s)}&\UK(\bB)}\ .$$}The dashed arrow is induced by the two embeddings $B\to B\oplus B$ and the universal property of the sum. \end{proof}

Let $S:\bA\to \bA$ be the endo-functor such that $\id_{\bA}\oplus \bS\simeq \bS$.
We get the equivalence  $\UK(\id_{\bA})+\UK(\bS)\simeq \UK(\bS)$ and this implies $\UK(\id_{\bA})\simeq 0$, hence
$\UK(\bA)\simeq 0$. \end{proof}

\begin{prop}\label{rgoigjoergergrgg}
The functor $\UK$ preserves filtered colimits. 
 \end{prop}
\begin{proof} The functor \eqref{ecfiuwhiuwefwfwe} clearly 
   preserves filtered colimits. 
The localization functor
\eqref{fiuzhriuferferferf} preserves filtered colimits.  
In order to see this we could use the model given in Remark \ref{oijoifwefewfwefwf} and observe that
{ha}mmock localization preserves filtered colimits, a filtered colimit  of Kan-complexes is again a Kan complex, and   the coherent nerve functor preserves filtered colimits, and finally that a filtered colimit of quasi-categories is a filtered colimit in the $\infty$-category $\Cat_{\infty}$.   
The inclusion $\Cat^{ex}_{\infty}\to \Cat_{\infty}$ {preserves} filtered colimits.  And
 finally, by definition the universal localizing $\cU_{\infty}$ invariant preserves filtered colimits.  \end{proof}

\section{A universal equivariant coarse homology theory}\label{foifioufouewfwefwefw}

%
%
%
%

\subsection{Coarse homology theories}

Following   \cite{Bunke:2016uq} a bornological coarse space is a triple $(X,\cC,\cB)$ of a set $X$ with a coarse structure $\cC$ and a bornology $\cB$ which are compatible in the sense that
every coarse thickening of a bounded set is again bounded. A morphism $f:(X,\cC,\cB)\to (X^{\prime},\cC^{\prime},\cB^{\prime})$ between bornological coarse spaces is a controlled and proper map. 
We thus have a category $\BC$ of bornological coarse spaces and morphisms.

For a group $G$ a $G$-bornological coarse space is a bornological coarse space $(X,\cC,\cB)$ with an action of $G$ by automorphisms such that its set of $G$-invariant entourages $\cC^{G}$ is cofinal in $\cC$. A morphism between $G$-bornological coarse spaces is an equivariant morphism of bornological coarse spaces. In this way we have defined a full subcategory $$G\BC\subseteq \Fun(BG,\BC)$$
of $G$-bornological coarse spaces and equivariant morphisms.

The category $G\BC$ has a symmetric monoidal structure which we denote by $\otimes$.
It is defined by
$$(X,\cC,\cB)\otimes (X^{\prime},\cC^{\prime},\cB^{\prime}):=(X\times X^{\prime},\cC\langle \cC\times \cC^{\prime}\rangle,\cB\langle \cB\times \cB^{\prime}\rangle)\ ,$$ where $G$ acts diagonally on the product set,
 $\cC\langle \cC\times \cC^{\prime}\rangle$ denotes the coarse structure generated by all products $U\times U^{\prime}$ of entourages of $X$ and $X^{\prime}$, and $ \cB\langle\cB\times \cB^{\prime}\rangle$ denotes the  bornology generated by all products $B\times B^{\prime}$ of bounded subsets of $X$ and $X^{\prime}$.

In   \cite{Bunke:2016uq} we axiomatized the notion of a coarse homology theory. The equivariant case will be studied throughly in \cite{bekw}. In the following   definition and the text below we describe the axioms.

Let $\cC$ be a stable cocomplete $\infty$-category and consider a functor
$$E:\BC\to \cC\ .$$

\begin{ddd}\label{foifjiofwefewwf}
$E$ is an equivariant  $\cC$-valued coarse homology theory if it satisfies:
\begin{enumerate}
\item coarse invariance.
\item excision.
\item vanishing on flasques.
\item $u$-continuity.
\end{enumerate}
 \end{ddd}

The detailed description of these properties is as  follows:
\begin{enumerate}
\item  (coarse invariance) We require that for every $G$-bornological coarse space $X$ the projection
$\{0,1\}_{max,max}\otimes X\to X$ induces an equivalence $$E(\{0,1\}_{max,max}\otimes X)\to E(X)\ .$$ Here for a set $S$ we let $S_{max,max}$ denote the set $S$ equipped with the maximal coarse and bornological structures.
\item (excision) A big family in a $G$-bornological coarse spaces  $X$ is a filtered family $(Y_{i})_{i\in I}$  of invariant subsets of $X$ such that for every $i$ in $I$ and entourage $U$ of $X$ there exists $j$ in $I$ such that $U[Y_{i}]\subseteq Y_{j}$.
We define $E(\cY):=\colim_{i\in I} E(Y_{i})$.  We furthermore set  $$E(X,\cY):=\Cofib(E(\cY)\to E(X))\ .$$
A complementary pair $(Z,\cY)$ consists of an invariant subset $Z$ of $X$ and a big family such that there exists $i$ in $I$ such that $Z\cup Y_{i}=X$. Note that $Z\cap \cY:=(Z\cap Y_{i})_{i\in I}$ is a big family in $Z$. 
Excision then requires that the canonical morphism
$$E(Z,Z\cap\cY)\to E(X,\cY)$$ is an equivalence. Here all subsets of $X$ are equipped with the induced bornological coarse structure.
\item A $G$-bornological coarse space $X$ is flasque if it admits an equivariant selfmap $f:X\to X$ (we say that $f$ implements flasquenss) such that $(f\times \id)(\diag_{X})$ is an entourage of $X$, for every entourage $U$ of $X$, the {subset}  $\bigcup_{n\in \nat} (f^{n}\times f^{n})(U)$ {of $X\times X$} is again an entourage
 of $X$, and for every bounded subset $B$ of $X$ there exists an $n$ in $\nat$ such that $f^{n}(X)\cap \Gamma B=\emptyset$. Vanishing on flasques requires that $$E(X)\simeq 0$$ if $X$ is flasque.
\item For every invariant entourage $U$ of $X$ we define a new bornological coarse space $X_{U}$. It has the underlying bornological space of $X$ and the coarse structure $\cC\langle U\rangle$. The condition of  $u$-continuity requires that the natural morphisms $X_{U}\to X$ induce an equivalence $$\colim_{U\in \cC^{G}}E(X_{U})\stackrel{\simeq}{\to}  E(X)\ .$$
 \end{enumerate}

\subsection{$X$-controlled $\bA$-objects}\label{reoifowefwefewfewf}

In this section we associate in a functorial way  to every additive category and bornological coarse  space
a new additive category of controlled objects and morphisms.
This construction is taken from \cite{bekw}.

\begin{ddd}
 A subset of a bornological coarse space is called locally finite if its intersection with every bounded subset is finite. \end{ddd}
 
Let $\bA$ be an additive category and $\hat \bA$ denote its $\Ind$-completion. We consider $\bA$ as a full subcategory of $\hat \bA$.

Let $X$ be a $G$-bornological coarse space.

\begin{ddd}\label{wefoiewfw}
An equivariant $X$-controlled $\bA$-object is a triple $(M,\phi,\rho)$, where
\begin{enumerate}
\item $M$ is an object of $\hat \bA$,
\item $\phi$ is a  measure on $X$ with values in the idempotents on $M$ admitting images,
\item $\rho$ is an action of $G$ on $M$,
\end{enumerate}
such that{:}
\begin{enumerate}
\item For every bounded subset $B$ of $X$ we  have $M(B)\in \bA$.
\item For every subset $Y$ of $X$ and element $g$ of $G$ we have
\begin{equation}\label{gpogjkpoegerg}
\phi(gY)=\rho(g)\phi(Y)\rho(g)^{-1} \ .
\end{equation}
\item  The natural map  \begin{equation}\label{rioiuiorugoieggergergerg}\bigoplus_{x\in X}M(\{x\})\to M \end{equation} is an isomorphism.
\item The subset $\supp(M,\phi,\rho):=\{x\in X\:|\: M(\{x\})\not=0\}$ of $X$ is locally finite.
\end{enumerate}
\end{ddd}

Here we write $M(Y)$ for the choice of an image of $\phi(Y)$.

\begin{rem}
The measure $\phi$ appearing in Definition \ref{wefoiewfw} is defined on all
subsets of $X$ and {is} finitely additive. For every two subsets $Y$ and $Z$ it satisfies $\phi(Y)\phi(Z)=\phi(Y\cap Z)$. In particular, we have $\phi(Y)\phi(Z)=\phi(Z)\phi(Y)$.   \end{rem}

Let $(M,\phi,\rho)$ and $(M^{\prime},\phi^{\prime},\rho^{\prime})$ be two equivariant $X$-controlled $\bA$-objects and $A:M\to M^{\prime}$ be a morphism in $\hat \bA$. Using the isomorphisms  \eqref{rioiuiorugoieggergergerg} for every pair of points $(x,y)$ in $X$ we obtain a morphism
$A_{x,y}:M(\{y\})\to  M^{\prime}(\{x\})$ between objects of $\bA$.
We say that $A$ is controlled if the set
$\supp(A):=\{(x,y)\in X\times X\:|\: A_{x,y}\not=0 \}$ is an entourage of $X$.

\begin{ddd}
A morphism $A:(M,\phi,\rho)\to (M^{\prime},\phi^{\prime},\rho^{\prime})$ between two equivariant  $X$-controlled $\bA$-objects is a morphism $A:M\to  M^{\prime}$ in $\hat \bA$ which intertwines the $G$-actions and is controlled.
\end{ddd}

We let $\bV_{\bA}^{G}(X)$ denote the category of equivariant $X$-controlled $\bA$-objects and morphisms. It is again an additive category. 
If $f:X\to X^{\prime}$ is an equivariant morphism between $G$-bornological coarse spaces, then we define a morphism between  additive categories
$$\bV_{\bA}^{G}(f):\bV_{\bA}^{G}(X)\to \bV_{\bA}^{G}(X^{\prime})\ , \quad (M,\phi,\rho)\mapsto f_{*}(M,\phi,\rho):=(M,\phi\circ f^{-1},\rho)\ , \quad A\mapsto f_{*}(A):=A\ .$$
In the way we have defined a functor
$$\bV_{\bA}^{G}:G\BC\to \Add\ .$$
If $G$ is the trivial group, then we omit $G$ from the notation.

\subsection{The universal equivariant coarse homology theory}

Let $\bA$ be an additive category.  \begin{ddd}\label{fiopfewfefwefwefff}
We define the functor
$$\UK\cX^{G}_{\bA}:=\UK\circ \bV^{G}_{\bA}:G\BC\to \cM_{loc}\ .$$
\end{ddd}

\begin{theorem}\label{wiwofewfewfewfewfw}
The functor $\UK\cX_{\bA}^{G}$ is an equivariant  coarse homology theory.
\end{theorem}
\begin{proof}
In the following four lemmas we verify the conditions listed in Definition \ref{foifjiofwefewwf}.
We keep the arguments sketchy as more details are given in \cite{bekw}.
\begin{lem}
$\UK\cX_{\bA}^{G}$ is coarsely invariant.
\end{lem}
\begin{proof}
If $X$ is a $G$-bornological coarse space, then the projection
$p:\{0,1\}_{max,max}\otimes X\to X$ induces an equivalence of additive categories
$\bV^{G}_{\bA}(\{0,1\}_{max,max}\otimes X)\to \bV^{G}_{\bA}(X)$. Using this fact the Lemma follows from Lemma \ref{roirujoigegergergg}. 

Let $i:X\to \{0,1\}_{max,max}\otimes X$ denote the inclusion given by the element $0$ of $\{0,1\}$. Then $p\circ i=\id_{X}$.  We now construct an isomorphism $u: \bV_{\bA}^{G}(i)\circ \bV_{\bA}^{G}(p)\to \id$. On the object
$(M,\phi,\rho)$ of $\bV_{\bA}^{G}(\{0,1\}_{max,max}\otimes X)$ it is given by $u_{(M,\phi,\rho)}:=\id_{M}:M\to M$.
This map defines an isomorphism of equivariant $\{0,1\}_{max,max}\otimes X$-controlled 
$\bA$-objects
$(M,\phi\circ p^{-1}\circ i^{-1},\rho)$ and $(M,\phi,\rho)$. 
One easily checks the conditions for a natural transformation.
\end{proof}

\begin{lem}
$\UK\cX_{\bA}^{G}$ is excisive.
\end{lem}
\begin{proof}

Let $X$ be a $G$-bornological coarse space and
  $\cY:=(Y_{i})_{i\in I}$ be a big family on $X$. Then
$$\bV^{G}_{\bA}(\cY):=\colim_{i\in I} \bV_{\bA}^{G}(Y_{i})$$ is a Karoubi-filtration in $\bV^{G}_{\bA}(X)$. Since $\UK$ commutes with filtered colimits we have an equivalence
$$\UK\cX^{G}_{\bA}(\cY)\simeq \UK(\bV^{G}_{\bA}(\cY))\ .$$ By Proposition \ref{roijoregergegr}
we get a fibre sequence $$\UK\cX^{G}_{\bA}(\cY)\to \UK\cX^{G}_{\bA}(X)\to \UK(\bV^{G}_{\bA}(X)/\bV^{G}_{\bA}(\cY))\ .$$
We conclude that \begin{equation}\label{regoijoigergegr}
\UK\cX^{G}_{\bA}(X,\cY)\simeq \UK(\bV^{G}_{\bA}(X)/\bV^{G}_{\bA}(\cY))\ .\end{equation}
Let now $(Z,\cY)$ be a   complementary pair. Then we get a similar sequence
$$\UK\cX^{G}_{\bA}(Z\cap\cY)\to \UK\cX^{G}_{\bA}(Z)\to \UK(\bV^{G}_{\bA}(Z)/\bV^{G}_{\bA}(Z\cap \cY))\ .$$  
One now checks that there is a canonical equivalence of additive categories
$$\psi:\bV^{G}_{\bA}(Z)/\bV^{G}_{\bA}(Z\cap \cY)\to \bV^{G}_{\bA}(X)/\bV^{G}_{\bA}(\cY)$$
induced by the inclusion $i:Z\to X$.
We define an inverse $$\phi:\bV^{G}_{\bA}(X)/\bV^{G}_{\bA}(\cY)\to \bV^{G}_{\bA}(Z)/\bV^{G}_{\bA}(Z\cap \cY)$$ as follows:
\begin{enumerate}
\item On objects $\phi$ sends $(M,\phi,\rho)$ to $(M(Z),\phi(Z\cap-),\rho_{|M(Z)})$.
\item On morphisms $\phi$ sends the class $[A]$ of $A:(M,\phi,\rho)\to (M^{\prime},\phi^{\prime},\rho^{\prime})$ to the class of $\phi^{\prime}(Z)A\phi(Z):M(Z)\to M^{\prime}(Z)$.
\end{enumerate}
One now observes that there is a canonical isomorphism
$$\phi\circ \psi\to \id_{\bV^{G}_{\bA}(Z)/\bV^{G}_{\bA}(Z\cap \cY)}\ .$$
Furthermore the inclusions
$M(Z)\to M$ induce an isomorphism $$\psi\circ \phi\to  \id_{ \bV^{G}_{\bA}(X)/\bV^{G}_{\bA}(\cY)}\ .$$ Using the equivalence \eqref{regoijoigergegr} twice we see that the natural morphism 
$$\UK\cX^{G}_{\bA}(Z,Z\cap \cY)\to \UK\cX^{G}_{\bA}(X,\cY)$$
is an equivalence. \end{proof}

\begin{lem}
$\UK\cX_{\bA}^{G}$  vanishes on flasques.
\end{lem}
\begin{proof}
Assume that $X$ is a flasque $G$-bornological coarse space. We claim that then $\bV_{\bA}^{G}(X)$ is a flasque additive category so that the Lemma follows from Proposition \ref{eroigerggerger}.  

Let $f:X\to X$ be the selfmap  implementing flasqueness. We can then define an exact functor
$$S:\bV_{\bA}^{G}(X)\to \bV_{\bA}^{G}(X)$$ by
$$S(M,\phi,\rho):=\Big(\bigoplus_{n\in \nat} M,\oplus_{n\in \nat} \phi\circ (f^{n})^{-1}, \oplus_{n\in \nat} \rho\Big)\ .$$
One checks that $S$ is well-defined and that 
$$\bV_{\bA}^{G}(f)\circ S\simeq \id_{\bV_{\bA}^{G}(X)}\oplus S\ .$$  There is an isomorphism $v:\bV_{\bA}^{G}(f)\to \id_{\bV_{\bA}^{G}(X)}$ which is given on the object 
$(M,\phi,\rho)$ by $$v_{(M,\phi,\rho)}:=\id_{M}:(M,\phi\circ f^{-1},\rho)\to (M,\phi,\rho)\ .$$
Consequently,
$$   S \simeq   \id_{\bV_{\bA}^{G}(X)}\oplus S $$ as required.
\end{proof}

  \begin{lem}
  $\UK\cX_{\bA}^{G}$ is $u$-continuous.
  \end{lem}
\begin{proof}
Let $X$ be a $G$-bornological coarse space. Then by an inspection of the definitions we have an isomorphism of additive categories
  $$\bV^{G}_{\bA}(X)\cong \colim_{U\in \cC^{G}} \bV^{G}_{\bA}(X_{U})\ .$$
  The lemma now follows from Proposition \ref{rgoigjoergergrgg}.
 \end{proof}
 \end{proof}

\subsection{Continuity}

In the section we verify an additional continuity property of the coarse homology $\UK\cX_{\bA}^{G}$. The property plays an important role in the comparison of assembly and forget control maps discussed in \cite{bekw}.

Let $X$ be a $G$-bornological coarse space.
  \begin{ddd}
A filtered family of invariant  subsets $(Y_{i})_{i\in I}$ of $X$ is called trapping if every locally finite subset  of $X$ is contained in $Y_{i}$  some index $i$ in $I$.\end{ddd}

\begin{ddd}
An equivariant  $\cC$-valued coarse homology theory $E$ is continuous if for every trapping exhaustion
$(Y_{i})_{i\in I}$ of a $G$-bornological coarse space $X$ the natural morphism
$ \colim_{i\in I}E(Y_{i})\to  E(X)$ is an equivalence.
\end{ddd}

\begin{prop}
 The coarse homology theory $\UK\cX^{G}_{\bA}$ is continuous.
\end{prop}
\begin{proof}
Let $X$ be a $G$-bornological coarse space.
If $(M,\phi,\rho)$ is an equivariant  $X$-controlled $\bA$-object, then its support $\supp(M,\phi,\rho)$   is locally finite.   

 Let $(Y_{i})_{i\in I}$ be a trapping
 exhaustion of $X$. 
 Then there exists $i$ in $I$ such that
 $(M,\phi,\rho)$ is supported on $Y_{i}$. 
 Consequently we have an isomorphism 
   $$\colim_{i\in I}\bV^{G}_{\bA}(Y_{i})\cong \bV^{G}_{\bA}(X)\ .$$
By  Proposition \ref{rgoigjoergergrgg}
 we now get
 $\colim_{i\in I}\UK\cX^{G}_{\bA}(Y_{i})\simeq \UK\cX^{G}_{\bA}(X)$. 
 \end{proof}
 
 \subsection{Functors on the orbit category}\label{fiowejofwefewfewfwf}

In this section we package the equivariant homology theories $\UK\cX_{\bA}^{H}$ for all subgroups $H$ of $G$ together into one object. We furthermore give an alternative construction of the equivariant coarse homology  theories from the non-equivariant  coarse homology theory.

The orbit category  $\Orb(G)$ of $G$ is the category of transitive $G$-sets and equivariant maps. Every object of the orbit category is isomorphic to one of the from $G/H$ with the left $G$-action for a subgroup $H$ of $G$. In particular we have the object
$G$ with the left action. The right action of $G$ induces an isomorphism of groups 
$\Aut_{\Orb(G)}(G)\cong G^{op}$. Hence we have a functor
$i:BG\to \Orb(G)^{op}$ which sends the unique object of $BG$ to $G$. If $\cC$ is a presentable $\infty$-category, then we have an adjunction
$$\Res:\Fun(\Orb(G)^{op},\cC)\leftrightarrows \Fun(BG,\cC):\Coind\ ,$$
where
$\Res$ is the restriction along $i$. 

For every subgroup $H$ of $G$ we have an evaluation
$$\Fun(\Orb(G)^{op},\cC)\to \cC\ , \quad E\mapsto E^{(H)}$$
at the object $G/H$. The family of all these evaluations detects limits, colimits, and equivalences in $\Fun(\Orb(G)^{op},\cC)$.

We consider the relative category $(\Add,W)$, where $W$  are the equivalences of additive categories. Applying the localization functor \eqref{fiuzhriuferferferf} we then obtain an $\infty$-category $\Add_{\infty}:=\cL(\Add,{W})$. We furthermore have a morphism   
$\ell:\Add\to \Add_{\infty}$ where we secretly identify ordinary categories with $\infty$-categories using the  nerve functor.

{A marked additive category is a   pair $(\bA,M)$ of an additive category $\bA$ and a collection $M$ of isomorphisms in $\bA$ (called marked isomorphisms) which is closed under composition and taking inverses. A morphism between marked additive categories $(\bA,M)\to (\bA^{\prime},M^{\prime})$ is an exact functor $\bA\to \bA^{\prime}$ which in addition sends $M \to M^{\prime}$.}
{We get a category $\Add^{+}$ of small marked additive categories and morphisms.
A marked isomorphism $u:F\to F^{\prime}$ between two functors $F,F^{\prime}:(\bA,M)\to (\bA^{\prime},M^{\prime})$ between marked additive categories  
is an isomorphism  of functors such that for every object $A$ of $\bA$ the isomorphism
$u_{A}:F(A)\to F(A^{\prime})$ is a marked isomorphism.}

{A morphism between marked additive categories is called a marked equivalence if it is invertible up to marked isomorphism.
We consider the relative category $(\Add^{+},W^{+})$, where $W^{+}$ are the marked equivalences. We define
$\Add^{+}_{\infty}:=\cL(\Add^{+},W^{+})$ and let $\ell^{+}:\Add^{+}\to \Add^{+}_{\infty}$ denote the canonical localization functor.
}

{The functor $\cF:\Add^{+}\to \Add$ which forgets the marking induces a commuting square \begin{equation}\label{ewfwijoi43t5ge}
\xymatrix{\Add^{+}\ar[r]^{\cF}\ar[d]^{\ell^{+}}&\Add\ar[d]^{\ell}\\\Add^{+}_{\infty}\ar[r]^{\cF_{\infty}}&\Add_{\infty}}
\end{equation}
of $\infty$-categories.
}

{In \cite[Sec. 7]{descent} it is shown that $\Add^{+}_{\infty}$ is a presentable $\infty$-category.
Consequently we have an  adjunction $$\Res:\Fun(\Orb(G)^{op},\Add^{+}_{\infty})\leftrightarrows \Fun(BG,\Add^{+}_{\infty}):\Coind\ .$$}

\begin{ex}\label{rgiojerogregergrege}
We consider a marked additive category $\bA^{+} $ with an action of $G$, i.e., an object of $\Fun(BG,\Add^{+})$. Then we can form the object
$\Coind(\ell^{+}(\bA^{+}))$ in $\Fun(\Orb(G)^{op},\Add^{+}_{\infty})$. In order to understand this object we calculate the evaluation $\Coind(\ell^{+}(\bA^{+}))^{(H)}$ for a subgroup $H$ of $G$. As a first step we have
$$\Coind(\ell^{+}(\bA^{+}))^{(H)}\simeq \lim_{BH} \ell^{+}(\bA^{+})\ .$$
We now analyse the right-hand side.
We consider the following marked additive category $\hat \bA^{+,H}$:
\begin{enumerate}
\item The objects of $\hat \bA^{+,H}$ are pairs
$(M,\rho)$, where $M$ is an object of $\bA^{+}$ and $\rho=(\rho(h))_{h\in H}$ is a collection of marked isomorphisms  $\rho(h):M\to h(M)$ such that $h(\rho(h^{\prime}))\circ \rho(h)=\rho(h^{\prime}h)$ for all pairs of elements $h^{\prime},h$ in $H$.
\item A morphism $A:(M,\rho)\to (M^{\prime},\rho^{\prime})$ in $\hat \bA^{+,H}$ is a morphism $A:M\to M^{\prime}$ in
$\bA^{+}$ such that $h(A)\circ \rho(h)=\rho^{\prime}(h)\circ A$ for every $h$ in $H$.
The composition of morphisms is induced by the composition in $\bA^{+}$.
\item The marked isomorphisms in $\hat \bA^{+,H}$ are the isomorphisms given by   marked isomorphisms in $\bA^{+}$.
\end{enumerate}
In \cite[Sec. 7]{descent} it is shown that we have   a natural equivalence 
$$ \lim_{BH} \ell^{+}(\bA^{+})\simeq \ell^{+}(\hat \bA^{+,H})\ .$$
As a final result we get the equivalence
\begin{equation}\label{hgfuiehiufwefwefwf}
\Coind(\ell^{+}(\bA^{+}))^{(H)}\simeq\ell^{+}(\hat \bA^{+,H})\ .
\end{equation}

\mbox{}
\end{ex}

{
In follows from \eqref{ewfwijoi43t5ge}, the universal property of the localization $\ell$, and Lemma \ref{roirujoigegergergg} that we have a factorization \begin{equation}\label{fweoifjwer345345}
\xymatrix{\Add^{+}\ar[r]^{\cF}\ar[dr]^{\ell^{+}}&\Add\ar[rr]^{\UK}\ar[dr]^{\ell}&&\cM_{loc} \\ &\Add^{+}_{\infty}\ar[r]^{\cF_{\infty}}&\Add_{\infty}\ar[ur]^{\UK_{\infty}}&}\ .
\end{equation}
}

{We refine the functor $\bV_{\bA}:\BC\to \Add$ described in Section \ref{reoifowefwefewfewf}
to a functor
$\bV_{\bA}^{+}:\BC\to \Add^{+}$ by defining
$$\bV_{\bA}^{+}(X):=(\bV_{\bA}(X), M)\ ,$$ where the set of  marked  isomorphisms $M$ is the set of $\diag(X)$-controlled isomorphisms in $\bV_{\bA}(X)$.
}

 \begin{ddd}\label{eiufwfwefewfewf}
 We define the functor $$\underline{\UK\cX_{\bA}}:G\BC\to \Fun(\Orb(G)^{op},\cM_{\infty})$$ as the composition
 \begin{eqnarray*}
 G\BC\to \Fun(BG,\BC)\stackrel{{\bV^{+}_{\bA}}}{\to}\Fun(BG,{\Add^{+}})\stackrel{{\ell^{+}}}{\to}\Fun(BG,{\Add^{+}_{\infty}}) &&\\ \stackrel{\Coind}{\to} \Fun(\Orb(G)^{op},\Add^{+}_{\infty}){\stackrel{\cF_{\infty}}{\to}  \Fun(\Orb(G)^{op},\Add_{\infty})} \stackrel{\UK_{\infty}}{\to} \Fun(\Orb(G)^{op},\cM_{loc})&&
 \end{eqnarray*}
 \end{ddd}

\begin{theorem}\label{wfoijwfwfewf}
$\underline{\UK\cX_{\bA}}$ is an equivariant coarse homology theory.
\end{theorem}
\begin{proof}
Since the conditions listed in Definition \ref{foifjiofwefewwf} concern equivalences or colimits and the collection of evaluations $\underline{\UK\cX_{\bA}}^{(H)}:G\BC\to \cM_{loc}$ detect equivalences and colimits it suffices to show that
$$ \underline{\UK\cX_{\bA}}^{(H)}:G\BC\to \cM_{loc}$$
are equivariant coarse homology theories for all subgroups $H$ of $G$.

We now note that if $E:H\BC\to \cC$ is a $H$-equivariant $\cC$-valued coarse  homology theory, then
$E\circ \Res^{G}_{H}:G\BC\to \cC$ is a  $G$-equivariant $\cC$-valued coarse homology theory, where
$\Res^{G}_{H}$ denotes the functor which restricts the action from $G$ to $H$.

The theorem now follows from the following Lemma  and Theorem \ref{wiwofewfewfewfewfw} ({applied to the subgroups $H$ in place of $G$}).\end{proof}

\begin{lem}
We have an equivalence of functors from $G{\BC}$ to $\cM_{loc}$
$$ \underline{\UK\cX_{\bA}}^{(H)}\simeq  \UK\cX^{H}_{\bA}\circ \Res^{G}_{H}\ .$$
\end{lem}
\begin{proof}
 Let $X$ be a $G$-bornological coarse space. {By  \eqref{hgfuiehiufwefwefwf} we have an equivalence   $$\Coind(\ell^{+}(\bV^{+}_{\bA}(X)))^{(H)}=\ell^{+}(\hat \bV^{+}_{\bA}(\Res^{G}_{H}(X))^{H})\ .$$} Hence we have an equivalence
{\begin{equation}\label{cknkjcdcdcdscscc}
 \underline{\UK\cX_{\bA}}^{(H)}(X)\simeq \UK_{\infty}(\cF_{\infty}(\ell^{+}(\hat \bV^{+}_{\bA}( \Res^{G}_{H}(X))^{H}))))\stackrel{\eqref{fweoifjwer345345}}{\simeq}   \UK(\cF(\hat \bV^{+}_{\bA} (\Res^{G}_{H}(X))^{H}))\ .
\end{equation}}By a comparison of the explicit description of $\hat \bV_{\bA}({\Res^{G}_{H}(X)})^{H}$ given in Example \ref{rgiojerogregergrege} with the definition of $\bV_{\bA}^{H}({\Res^{G}_{H}(X)})$ given in Section \ref{reoifowefwefewfewf}, {and using that \eqref{gpogjkpoegerg} expresses exactly the condition that $\rho(g)$ is $\diag(X)$-controlled (i.e., marked),} one now gets a canonical isomorphism (actually an equality) {of additive categories}
$${\cF(\hat \bV^{+}_{\bA}( \Res^{G}_{H}(X))^{H})\cong \bV^{H}_{\bA}( \Res^{G}_{H}(X))}\ .$$ 
In view of Definition \ref{fiopfewfefwefwefff} we have an equivalence 
$$\UK({\cF(\hat \bV^{+}_{\bA} (\Res^{G}_{H}(X))^{H})})\simeq \UK\cX^{H}_{\bA}({\Res^{G}_{H}(X)})\ .$$ 
{Together with \eqref{cknkjcdcdcdscscc} it yields the equivalence
$$ \underline{\UK\cX_{\bA}}^{(H)}(X)\simeq  \UK\cX^{H}_{\bA}(\Res^{G}_{H}(X))\ .$$ as desired.}
\end{proof}

\bibliographystyle{alpha}
\bibliography{dec2}

\end{document}